\documentclass[oneside]{amsart}

\usepackage[utf8]{inputenc}
\usepackage{amsthm,mathtools,tikz,float,caption}
\usetikzlibrary{patterns,matrix}
\usepackage[protrusion=true]{microtype}
\usepackage{breakurl}
\usepackage{hyperref}

\usepackage[natbib=true,style=numeric,maxnames=10]{biblatex}
\usepackage[babel]{csquotes}
\bibliography{paper-generic-freeness.bib}

\title[Grothendieck's generic freeness lemma]{An elementary and constructive
proof of Grothendieck's generic freeness lemma}
\author{Ingo Blechschmidt}
\address{Max Planck Institute for Mathematics in the Sciences \\
Inselstraße 22 \\
04103 Leipzig, Germany}
\email{ingo.blechschmidt@mis.mpg.de}

\theoremstyle{definition}
\newtheorem{defn}{Definition}[]

\theoremstyle{plain}
\newtheorem{prop}[defn]{Proposition}

\newtheorem{lemma}[defn]{Lemma}
\newtheorem{thm}[defn]{Theorem}

\theoremstyle{remark}

\newcommand{\aaa}{\mathfrak{a}}

\newcommand{\mmm}{\mathfrak{m}}
\newcommand{\I}{\mathcal{I}}
\newcommand{\J}{\mathcal{J}}
\newcommand{\E}{\mathcal{E}}
\newcommand{\B}{\mathcal{B}}
\renewcommand{\O}{\mathcal{O}}
\newcommand{\defeq}{\vcentcolon=}
\DeclareMathOperator{\Spec}{Spec}

\newcommand{\stacksproject}[1]{\cite[{\href{https://stacks.math.columbia.edu/tag/#1}{Tag~#1}}]{stacks-project}}

\begin{document}

\begin{abstract}
  We present a new and direct proof of Grothendieck's generic freeness
  lemma in its general form. Unlike the previously published proofs, it does not
  proceed in a series of reduction steps and is fully constructive, not
  using the axiom of choice or even the law of excluded middle. It was
  found by unwinding the result of a general topos-theoretic technique.
\end{abstract}

\maketitle
\thispagestyle{empty}

\noindent
We prove Grothendieck's generic freeness lemma in the
following form.

\begin{thm}\label{thm:algebraic}Let~$A$ be a reduced ring (commutative, with unit). Let~$B$ be
an~$A$-algebra of finite type. Let~$M$ be a finitely generated~$B$-module.
If~$f = 0$ is the only element of~$A$ such that
\begin{enumerate}
\item the~$A[f^{-1}]$-modules $B[f^{-1}]$ and $M[f^{-1}]$ are free,
\item the~$A[f^{-1}]$-algebra~$B[f^{-1}]$ is of finite presentation and
\item the~$B[f^{-1}]$-module~$M[f^{-1}]$ is finitely presented,
\end{enumerate}
then~$1 = 0$ in~$A$.
\end{thm}

Previously known proofs either only cover the case where~$A$ is a Noetherian
integral domain, where one can argue by \emph{dévissage} (see for
instance~\cite[Lemme~6.9.2]{ega-4-2},
\cite[Thm.~24.1]{matsumura:commutative-ring-theory} or
\cite[Thm.~14.4]{eisenbud:commutative-algebra}), or proceed in a series of
intermediate steps, reducing to that case (see for
instance~\cite{staats:generic-freeness} or~\stacksproject{051Q}); but in fact,
a direct proof is possible and shorter. The new proof unveils a certain
combinatorial aspect to Grothendieck's generic freeness lemma, does not require
any advanced prerequisites in commutative algebra and does not use the axiom of
choice or the law of excluded middle. It is purely element-based, not referring
to ideals of~$A$, and doesn't use Noether normalization.

Grothendieck's generic freeness lemma is often presented in contrapositive form
or in the following geometric variant:

\begin{thm}\label{thm:geometric}Let~$A$ be a reduced ring. Let~$B$ be
an~$A$-algebra of finite type. Let~$M$ be a finitely generated~$B$-module. Then
the space~$\Spec(A)$ contains a dense open~$U$ such that over~$U$,
\begin{enumerate}
\item[(a)] $B^\sim$ and~$M^\sim$ are locally free as sheaves of~$A^\sim$-modules,
\item[(b)] $B^\sim$ is of finite presentation as a sheaf of~$A^\sim$-algebras and
\item[(c)] $M^\sim$ is finitely presented as a sheaf of~$B^\sim$-modules.
\end{enumerate}
\end{thm}

Theorem~\ref{thm:geometric} immediately follows from
Theorem~\ref{thm:algebraic} by defining~$U$ as the union of all those basic
opens~$D(f)$ such that~(1),~(2) and~(3) hold. It is clear that~(a),~(b)
and~(c) hold over~$U$, and~$U$ is dense for if~$V$ is an arbitrary open
such that~$U \cap V = \emptyset$, the open~$V$ is itself empty: Let~$h \in A$ be
such that~$D(h) \subseteq V$. The hypothesis implies the assumptions of
Theorem~\ref{thm:algebraic} for the datum~$(A[h^{-1}], B[h^{-1}], M[h^{-1}])$.
Thus~$1 = 0 \in A[h^{-1}]$, so~$h$ is nilpotent and~$D(h) = \emptyset$.

The new proof was found using a general topos-theoretic technique which we
believe to be useful in other situations as well. This technique allows to view
reduced rings and their modules from a different point of view, one from which
reduced rings look like fields. Since Grothendieck's generic freeness
is trivial for fields, this technique yields a trivial proof for reduced rings.
The proof presented here was obtained by unwinding the topos-theoretic proof,
yielding a self-contained argument without any references to topos theory.
We refer readers who want to learn about this technique to a forthcoming companion 
paper~\cite{blechschmidt:wlog}.

\textbf{Acknowledgments.} The proof presented here was prompted by a question by
user~HeinrichD on MathOverflow~\cite{mo:kernel} and greatly
benefited from discussions with Martin Brandenburg, who employed the constructive
version in a paper of his~\cite{brandenburg:schur}. We are grateful to Thierry
Coquand and Peter Schuster for valuable advice, to Giuseppe Rosolini for
comments regarding the presentation of the paper, and to Marc
Nieper-Wißkirchen for carefully guiding our PhD studies~\cite{blechschmidt:phd} at the University of
Augsburg, where most of the work for this paper was carried out.

\section{The proof of the finitely-generated case}

The following proposition is just a special instance of Grothendieck's generic
freeness lemma. Its proof is easier and shorter than the proof of the general
case, which is why we present it here. The general proof
will not refer to this one.

\begin{prop}Let~$A$ be a reduced ring. Let~$M$ be a finitely
generated~$A$-module. If~$f = 0$ is the only element of~$A$ such
that~$M[f^{-1}]$ is a finite free~$A[f^{-1}]$-module, then~$1 = 0$ in~$A$.
\end{prop}

\begin{proof}We proceed by induction on the length of a given generating family
of~$M$. Let~$M$ be generated by~$(v_1,\ldots,v_m)$.

We show that the family~$(v_1,\ldots,v_m)$ is linearly independent. Let~$\sum_i
a_i v_i = 0$. Over~$A[a_i^{-1}]$, the vector~$v_i \in M[a_i^{-1}]$ is a linear
combination of the other generators. Thus~$M[a_i^{-1}]$ can be generated as
an~$A[a_i^{-1}]$-module by fewer than~$m$ generators. The induction hypothesis,
applied to this module, yields that~$1 = 0$ in~$A[a_i^{-1}]$. Since~$A$ is
reduced, this amounts to~$a_i = 0$.

We finish by using the assumption for~$f = 1$.
\end{proof}

We remark that the proof takes a somewhat curious course: Our goal is to
verify~$1 = 0$, but as an intermediate step we verify that~$M$ is free, which
after the fact will be a trivial statement. The general proof in the next
section will have a similar style. This approach is reminiscient of
Richman's uses of trivial rings~\cite{richman:trivial-rings}.

\section{The proof of the general case}

\begin{proof}[Proof of Theorem~\ref{thm:algebraic}]
Let~$B$ be generated by~$(x_1,\ldots,x_n)$ as an~$A$-algebra and
let~$M$ be generated by~$(v_1,\ldots,v_m)$ as a~$B$-module. We endow the sets
\begin{align*}
  \I &\defeq \{ (i_1,\ldots,i_n) \,|\, i_1,\ldots,i_n \geq 0 \}
  \quad\text{and} \\
  \J &\defeq \{ (\ell, i_1,\ldots,i_n) \,|\, 1 \leq \ell \leq m, i_1,\ldots,i_n \geq 0 \}
\end{align*}
with the lexicographic order. The family~$(w_J)_{J \in \J} \defeq (x_1^{i_1}
\cdots x_n^{i_n} v_\ell)_{(\ell,i_1,\ldots,i_n) \in \J}$ generates~$M$ as
an~$A$-module, and we will call a subfamily~$(w_J)_{J \in \J' \subseteq \J}$
\emph{good} if and only if for all~$J \in \J$, the vector~$w_J$ is a linear
combination of the vectors~$(w_{J'})_{J' \in \J', J' \preceq J}$,
and if~$(\ell,i_1,\ldots,i_n) \not\in \J'$ implies~$(\ell,k_1,\ldots,k_n) \not\in
\J'$ for all~$k_1 \geq i_1, \ldots, k_n \geq i_n$. Figure~\ref{fig:good} shows
how a good generating family can look like.
Similarly, we define when a subfamily of the canonical generating
family~$(x_1^{i_1} \cdots x_n^{i_n})_{(i_1,\ldots,i_n) \in \I}$ of~$B$ is good
(which is just the special case~$m = 1$).

We then proceed by induction on the
shapes of a given good generating family~$(w_J)_{J \in \J'}$ for~$M$ and a given
good generating family~$(s_I)_{I \in \I'}$ for~$B$, starting with the canonical ones. It is
reasonably obvious that this induction is well-founded; the formal statement
that it is so is known as \emph{Dickson's Lemma} (see, for
instance,~\cite[Thm.~1.1]{veldman:kruskal}).

We show that~$(w_J)_{J \in \J'}$ is a basis of~$M$ by verifying linear
independence. Thus let~$\sum_J a_J w_J = 0$ in~$M$. We show that all
coefficients in this sum are zero, starting with the largest
appearing index~$J$: In the module~$M[a_J^{-1}]$ over the localized ring~$A[a_J^{-1}]$,
the vector~$w_J$ is a linear combination of generators with
smaller index. Removing~$w_J = x_1^{i_1} \cdots x_n^{i_n} v_\ell$ and also all
vectors~$x_1^{k_1} \cdots x_n^{k_n} v_\ell$ where~$k_1 \geq i_1, \ldots,
k_n \geq i_n$, we obtain a subfamily which is still good for
the localized module. The induction hypothesis, applied
to~$A[a_J^{-1}]$ and its module~$M[a_J^{-1}]$, therefore implies that $A[a_J^{-1}] = 0$.
Thus $a_J = 0$ since~$A$ is reduced.

Similarly, we show that the given good generating family~$(s_I)_{I \in \I'}$ is
a basis. Thus~$M$ and~$B$ are free over~$A$. We fix for any corner~$J$
of~$\J'$, as indicated in Figure~\ref{fig:good},
a way of expressing~$w_J = \sum_K a_{JK} w_K$ as a linear combination of
generators with strictly smaller index. Let~$\widehat w_{(\ell,i_1,\ldots,i_n)}
\defeq x_1^{i_1} \cdots x_n^{i_n} V_\ell$ in the free~$B$-module~$B\langle
V_1,\ldots, V_m\rangle$. The canonical map
\[ \widehat M \defeq B\langle V_1,\ldots, V_m\rangle/
  (\widehat w_J - \textstyle\sum_K a_{JK} \widehat w_K)_{\text{$J$ corner of~$\J'$}} \longrightarrow M
\]
is trivially well-defined and surjective. It is also injective, since any
element of~$\widehat M$ can be written as an~$A$-linear combination of the
vectors~$(\widehat w_J)_{J \in \J'}$ by employing the corner relations a finite
number of times.  Therefore~$M$ is finitely presented as a~$B$-module.

In a similar vein, a quotient algebra of~$A[X_1,\ldots,X_n]$, where
we mod out by a suitable ideal with as many generators as corners of~$\I'$, is
isomorphic to~$B$. Thus~$B$ is finitely presented as an~$A$-algebra.

We finish by using the assumption for~$f = 1$.
\end{proof}

\begin{figure}[t]
  \captionsetup{width=.9\linewidth}
  \begin{tikzpicture}[scale=0.9]
    \draw[step=1cm,gray,very thin] (0,0) grid (8,8);

    \draw
      (0,8) -- (0,0) -- (8,0);

    \fill[pattern=north east lines,pattern color=gray,opacity=0.5]
      (2,8) -- (2,5) -- (5,5) -- (5,3) -- (6,3) -- (6,2) -- (8,2) -- (8,8);
    \draw
      (2,8) -- (2,5) -- (5,5) -- (5,3) -- (6,3) -- (6,2) -- (8,2);

    \fill[fill=black] (2,5.25) -- (2,5) -- (2.25,5) -- (2.25,5.25);
    \fill[fill=black] (5,3.25) -- (5,3) -- (5.25,3) -- (5.25,3.25);
    \fill[fill=black] (6,2.25) -- (6,2) -- (6.25,2) -- (6.25,2.25);

    \fill[fill=blue!40!white,opacity=0.5]
      (0,8) -- (0,0) -- (4,0) -- (4,3) -- (3,3) -- (3,5) -- (2,5) -- (2,8);

    \fill[fill=red!40!white,opacity=0.5]
      (3,3) -- (4,3) -- (4,4) -- (3,4);

    \begin{scope}[yshift=4cm, xshift=4cm]
      \matrix[matrix of nodes,nodes={inner sep=0pt,text width=0.9cm,align=center,minimum height=0.9cm}]{
        \tiny $x^0y^7v_1$ & \tiny $x^1y^7v_1$ & \tiny $x^2y^7v_1$ & \tiny $x^3y^7v_1$ & \tiny $x^4y^7v_1$ & \tiny $x^5y^7v_1$ & \tiny $x^6y^7v_1$ & \tiny $x^7y^7v_1$ & \\
        \tiny $x^0y^6v_1$ & \tiny $x^1y^6v_1$ & \tiny $x^2y^6v_1$ & \tiny $x^3y^6v_1$ & \tiny $x^4y^6v_1$ & \tiny $x^5y^6v_1$ & \tiny $x^6y^6v_1$ & \tiny $x^7y^6v_1$ & \\
        \tiny $x^0y^5v_1$ & \tiny $x^1y^5v_1$ & \tiny $x^2y^5v_1$ & \tiny $x^3y^5v_1$ & \tiny $x^4y^5v_1$ & \tiny $x^5y^5v_1$ & \tiny $x^6y^5v_1$ & \tiny $x^7y^5v_1$ & \\
        \tiny $x^0y^4v_1$ & \tiny $x^1y^4v_1$ & \tiny $x^2y^4v_1$ & \tiny $x^3y^4v_1$ & \tiny $x^4y^4v_1$ & \tiny $x^5y^4v_1$ & \tiny $x^6y^4v_1$ & \tiny $x^7y^4v_1$ & \\
        \tiny $x^0y^3v_1$ & \tiny $x^1y^3v_1$ & \tiny $x^2y^3v_1$ & \tiny $x^3y^3v_1$ & \tiny $x^4y^3v_1$ & \tiny $x^5y^3v_1$ & \tiny $x^6y^3v_1$ & \tiny $x^7y^3v_1$ & \\
        \tiny $x^0y^2v_1$ & \tiny $x^1y^2v_1$ & \tiny $x^2y^2v_1$ & \tiny $x^3y^2v_1$ & \tiny $x^4y^2v_1$ & \tiny $x^5y^2v_1$ & \tiny $x^6y^2v_1$ & \tiny $x^7y^2v_1$ & \\
        \tiny $x^0y^1v_1$ & \tiny $x^1y^1v_1$ & \tiny $x^2y^1v_1$ & \tiny $x^3y^1v_1$ & \tiny $x^4y^1v_1$ & \tiny $x^5y^1v_1$ & \tiny $x^6y^1v_1$ & \tiny $x^7y^1v_1$ & \\
        \tiny $x^0y^0v_1$ & \tiny $x^1y^0v_1$ & \tiny $x^2y^0v_1$ & \tiny $x^3y^0v_1$ & \tiny $x^4y^0v_1$ & \tiny $x^5y^0v_1$ & \tiny $x^6y^0v_1$ & \tiny $x^7y^0v_1$ & \\
      };
    \end{scope}
  \end{tikzpicture}
  \caption{\label{fig:good}A graphical depiction of a good generating family (the non-hatched cells) in
  the special case~$n = 2, m = 1$, writing~``$x$'' and~``$y$'' for~$x_1$ and~$x_2$. The hatched cells indicate vectors which
  have already been removed from the family. The small black squares indicate
  \emph{corners}. If the vector in the red cell will be found to be expressible
  as a linear combination of vectors with smaller index (blue cells), it will
  be removed, along with the vectors in all cells to the top and to the right
  of the red cell.}
\end{figure}

\section{Conclusion}

Commutative algebra abounds with techniques which allow us to reduce quite
general situations to easier ones. These techniques often yield short and slick
proofs; however, they come at an expense: They are typically nonconstructive in
nature, employing for instance the axiom of choice, and do not argue using only
the data at hand, but using additional auxiliary objects such as maximal
ideals. We feel that once a subject is better understood, it is desirable to
have more informative, direct proofs available which illuminate the
proven claims more clearly; similar as to how bijective proofs are preferred over
calculational inductive ones in combinatorics.

Let us consider as a specific example the statement that the existence of a linear
surjection~$A^n \to A^m$ with~$n < m$ between finite free modules over an
arbitrary ring~$A$ implies~$1 = 0 \in A$. The standard proof of this fact proceeds
by contradiction and passes to the quotient~$A/\mmm$, where~$\mmm$ is a maximal
ideal of~$A$, thereby reducing to the situation that the ring is a field. In
contrast, a direct proof such as Richman's~\cite{richman:trivial-rings}
refers only to objects mentioned in the statement itself and
explicitly tells us how to deduce the equation~$1 = 0$ from the~$m$ equations
which express that each basis vector of~$A^m$ has a preimage.

In a similar fashion, the new proof of Grothendieck's generic freeness lemma
explicitly tells us how to deduce~$1 = 0$ from the given conditional
equations expressing that~$f = 0$ is the only element with properties~(1),~(2)
and~(3). The history of Grothendieck's generic freeness lemma goes back more
than fifty years; we are slightly surprised that a direct proof was
discovered only now.

Direct proofs sometimes generalize to new situations where the reduction techniques
employed by more abstract proofs cannot be applied. This is the case for
Grothendieck's generic freeness lemma, which allows for the following
generalization:

\begin{thm}\label{thm:general-generic-freeness}
Let~$(X,\O_X)$ be a ringed space (or ringed locale, or ringed topos), such
that for every local section~$s$ of~$\O_X$, if the only open on which~$s$ is
invertible is the empty one, then~$s = 0$. Let~$\B$ be a sheaf
of~$\O_X$-algebras of finite type. Let~$\E$ be a sheaf of~$\B$-modules of finite type. Then there is a dense
open~$U$ such that over~$U$,
\begin{enumerate}
\item[(a)] $\B$ and~$\E$ are locally free as sheaves of~$\O_X$-modules,
\item[(b)] $\B$ is of finite presentation as a sheaf of~$\O_X$-algebras and
\item[(c)] $\E$ is finitely presented as a sheaf of~$\B$-modules.
\end{enumerate}
\end{thm}

\begin{proof}Our proof of Theorem~\ref{thm:algebraic} can be easily adapted to
this more general setting. Where that proof concludes~$a_J = 0$ for ring
elements~$a_J$ by considering the localized situation~$A[a_J^{-1}]$, we now
conclude~$a_J = 0$ for local sections~$a_J$ of~$\O_X$ by considering the
situation over the restriction to~$D(a_J)$. Whereas before this type of argument was
powered by the reducedness assumption on~$A$, it is not supported by the
assumption on~$X$.
We omit further details.
\end{proof}

The spectrum of a ring~$A$ is a space of the kind required by
Theorem~\ref{thm:general-generic-freeness} if and only if~$A$ is reduced,
thus Theorem~\ref{thm:general-generic-freeness} indeed
generalizes Theorem~\ref{thm:geometric}. Further examples for admissible spaces
are given by any topological, smooth or complex manifold; a corollary
of Theorem~\ref{thm:general-generic-freeness} for these examples is that
quotients of vector bundles, computed in the category of sheaves of modules,
are again vector bundles after restricting to suitable dense opens.

It is hard to say with certainty that a given proof does not generalize to a
new situation; but we do not see how this could be the case for the proofs
cited in the introduction.

\printbibliography

\end{document}